\documentclass[oneside,a4paper,10pt]{amsart}

\newcommand{\thetitle}{The expected number of eigenvalues of a real gaussian tensor}

\newcommand{\thelanguage}{english}

\usepackage{geometry}
\usepackage[utf8]{inputenc}
\usepackage[onehalfspacing]{setspace}
\usepackage[T1]{fontenc}
\usepackage[\thelanguage]{babel}
\usepackage[babel]{csquotes}
\usepackage{tikz}
\usepackage{tikz-3dplot}
\usepackage{mathrsfs}
\usepackage{rotating}
\usepackage{float}
\usepackage{amsthm}
\usepackage{amsmath,amssymb}
\usepackage[colorlinks=true, citecolor=blue]{hyperref}
\usepackage[\thelanguage]{cleveref}
\usepackage{mathdots}
\usepackage{color,calc,comment}
\usepackage{multirow}
\usepackage{enumitem}
\usepackage{tabularx}
\usepackage{graphicx}
\usepackage{nicefrac}
\usepackage[all,2cell]{xy}  
\usepackage{mathtools}
\usepackage{stmaryrd} 
\usepackage{ifpdf}
\usepackage[h]{esvect}
\usepackage{bbm}
\usepackage{marvosym}
\usepackage{blkarray}
\usepackage{ifdraft}
\usepackage{chngcntr}
\usepackage{listings}
\usepackage{fancyhdr}
\usepackage{algorithm,algorithmic}
\usepackage{polynom}

\newcommand{\HC}{\mathbb{C}}

\newcommand{\HR}{\mathbb{R}}
\newcommand{\HS}{\mathbb{S}}

\newcommand{\cA}{\mathcal{A}}

\newcommand{\cC}{\mathcal{C}}

\newcommand{\cE}{\mathcal{E}}

\newcommand{\cH}{\mathcal{H}}

\newcommand{\cO}{\mathcal{O}}

\newcommand{\set}[1]{\left\{#1\right\}}
\newcommand{\cset}[2]{\left\{#1\mid #2\right\}}
\newcommand{\norm}[1]{\left\lvert #1 \right\rvert}
\newcommand{\Norm}[1]{\left\lVert #1 \right\rVert}
\newcommand{\rank}{\mathrm{rk}\,}
\newcommand{\id}{\mathrm{id}}
\renewcommand{\d}{\mathrm{ d}}

\DeclareMathOperator*{\mean}{\mathbb{E}}

\newcommand{\hV}{{\bf V}}
\newcommand{\hW}{{\bf W}}

\newcommand\restr[2]{\ensuremath{\left.#1\right|_{#2}}}

\numberwithin{equation}{section}
\numberwithin{figure}{section}

\theoremstyle{plain}
\newcounter{numbering} \numberwithin{numbering}{section}
\newtheorem{thm}[numbering]{Theorem}
\newtheorem{lemma}[numbering]{Lemma}
\newtheorem{prop}[numbering]{Proposition}

\theoremstyle{definition}

\theoremstyle{remark}

\newtheorem*{rem}{Remark}
\theoremstyle{plain}

\crefname{equation}{}{}
\crefname{equation}{}{}
\crefname{figure}{Figure}{Figures}
\crefname{section}{Section}{Sections}
\crefname{lemma}{Lemma}{Lemmata}
\crefname{prop}{Proposition}{Propositions}
\crefname{thm}{Theorem}{Theorems}
\crefname{cor}{Corollary}{Corollaries}
\crefname{dfn}{Definition}{Definitions}
\crefname{notation}{Notations}{Notations}
\crefname{rem}{Remark}{Remarks}
\crefname{claim}{Claim}{claims}
\begin{document}
\title{\thetitle}
\author{Paul Breiding}
\thanks{Institute of Mathematics, Technische Universität Berlin, breiding@math.tu-berlin.de. Partially supported by DFG research grant BU 1371/2-2.}
\begin{abstract}
\vspace{-0.2cm}
In this paper we compute the expected number of eigenvalues of a random real tensor $A \in(\HR^{n})^{\otimes(d+1)}$, whose entries are centered gaussian random variables with variance $\sigma^2= 1$.
\end{abstract}
\maketitle
\vspace{-0.5cm}
\smallskip

\vspace{-0.2cm}
\section{Introduction}
Eigenvalues of tensors are a direct generalization of the concept of eigenvalues of matrices and, as the latter, have potential for a widespread amount of applications, see \cite{lim2}. For instance, the largest Z-eigenvalue \cite{qi3,qi2,qi1} of a symmetric tensor is linked to the \emph{best symmetric rank-one approximation} \cite{rank-one-approx, low-rank-approx, rank-one-approx2}, which is applicable to deflation methods in tensor decomposition \cite{rank-one-approx3}, a tool used in many areas such as blind source separation, data compression, imaging or genomic data analysis, see \cite{low-rank-approx, tensor_decomp, lathauwer-rank-one, rank-one-approx2}.

In \cite{draisma-horobet} Draisma and Horobet point out that application-oriented algorithms to compute best rank-one approximations or, more generally, best low-rank approximations are mostly of local nature. Therefore they face difficulties when close to non-minimal critical points of the distance function that one wants to minimize. This motivates them to define the number of those critical points as a complexity measure for the low-rank approximation problem and to ask for the \emph{average number of critical points} when the entries of the tensor are real gaussian random variables.

In this paper we answer a related question: \emph{What is the average number of eigenvalues of a tensor whose entries a real gaussian random variables?} Both question are related insofar as the number of critical points of the distance function from a real symmetric tensor $A$ to the set of symmetric rank-one-tensors equals the number of eigenvalues of $A$. In this article, however, $A$ is not necessarily symmetric. Nevertheless, we understand this paper as a contribution to getting more insight into Draisma's and Horobet's problem and, ultimately, more insight into the nature of rank-one approximation of real tensors.

Suppose that $A=(A_{i_0,i_1,\ldots,i_{d}}) \in(\HC^{n})^{\otimes(d+1)}$ is an order $d+1$ tensor of format $n\times \ldots \times n$. Then $A$ defines a multilinear map
	\begin{equation}\label{multilinear_map}
	\hat{A}:\underbrace{\HC^n \times \ldots \times \HC^n}_{d\text{ many times}} \to \HC^n,\; (v_1,\ldots,v_d)\mapsto (A(e_1,v_1,\ldots,v_d),\ldots, A(e_n,v_1,\ldots,v_d)),
	\end{equation}
where $e_1,\ldots,e_n$ is the standard basis in $\HC^n$. A pair $(v,\lambda)\in(\HC^n\backslash\set{0})\times \HC$ is called an \emph{eigenpair} of $A$, if  
	\begin{equation}\label{1}
	Av^d := \hat{A}(v,\ldots,v) = \lambda v.
	\end{equation}
Since for all $t\in\HC^\times$ we have $Av^d=\lambda v$, if and only if $A(tv)^d=(t^{d-1}\lambda)(tv)$, eigenpairs define points in a weighted projective space. Cartwright and Sturmfels \cite{sturmfels-cartwright} call $(v,\lambda)$, $(tv,t^{d-1}\lambda)$ \emph{equivalent} and they identify the number of equivalence classes of eigenpairs for a generic tensor $A\in(\HC^{n})^{\otimes(d+1)}$ as $D(n,d):=\sum_{i=0}^{n-1} d^i$. 

When $A$ is real, eigenpairs are invariant under complex conjugation; \emph{real eigenpairs} are eigenpairs $(v,\lambda)\in(\HR^n\backslash\set{0})\times \HR$.  Let us denote the number of equivalence classes of real eigenpairs of $A$ by $\#_\HR(A)$. 

A related notion are \emph{Z-eigenvalues} defined by Qi \cite{qi3,qi2,qi1}: If $(v,\lambda)$ is a real eigenpair of $A$, Qi calls the number $\lambda$ a Z-eigenvalue of $A$, if $v^Tv=1$. We have the equality
	\begin{equation*}
	\#_\HR(A) = \begin{cases}
	\frac{1}{2} \; \#\set{\text{ Z-eigenvalues of } A}, & \text{ if } d \text{ is even and $0$ is not an eigenvalue.} \\
	\frac{1}{2} \; \left(\#\set{\text{ Z-eigenvalues of } A} + 1\right), & \text{ if } d \text{ is even and $0$ is an eigenvalue.} \\ \#\set{\text{ Z-eigenvalues of } A}, &\text{ if } d \text{ is odd.}  \end{cases}
	\end{equation*}

Unlike for complex numbers there is no \emph{generic number} of real eigenpairs: although the function $A\mapsto\#_\HR(A)$ is constant on some open semi-algebraic subsets of $(\HR^n)^{\otimes (d+1)}$, it 'jumps' when crossing a real-codimension-one discriminant; compare \cref{fig1}. This observation motivates the probabilistic study of the eigenpair problem for real tensors presented in this article.

A natural requirement is to consider a probability distribution on $(\HR^n)^{\otimes (d+1)}$ (a space of maps!) that is invariant under an orthogonal change of variables. Note that equation \cref{1} is invariant under such a change of variables. A possible choice for this probability distribution is given by requiring that the components $A_{i_0,i_1,\ldots,i_{d}}$ are independent and identically distributed centered gaussian random variables with variance $\sigma^2=1$:
\begin{equation}\label{def_gaussian_A}
A_{i_0,i_1,\ldots,i_{d}} \sim N(0,1).
\end{equation}
In what follows we call a real tensor $A=(A_{i_0,i_1,\ldots,i_{d}})$ that has the distribution from \cref{def_gaussian_A} \emph{gaussian}.

\begin{rem}
There are other possible choices (indeed there is a whole continuous family of choices) of probability distributions on $(\HR^n)^{\otimes (d+1)}$ that are invariant under an orthogonal change of variables, but the current choice is especially interesting because it allows to make comparisons between expectations over the reals and generic answers over the complex numbers. This is due to the fact that the generic answer over the complex numbers can still be obtained as an expectation, with respect to the unique centered gaussian distribution, which is invariant under unitary change of variables and defined as in \cref{def_gaussian_A} but taking complex gaussians.
\end{rem}

For $d=1$, the matrix case, the expected value of $\#_\HR(A)$ for a real gaussian matrix was computed in~\cite{real_eigenvalues} and our paper is very much inspired by this work. For the sake of completeness we will include the results from \cite{real_eigenvalues} in our main theorem, \cref{main_thm}.	
	
Eigenpairs of tensors are nothing but eigenpairs of homogeneous polynomial systems, which we defined in \cite{homotopy_eigenpairs, distr}. To be precise, let $\cH_d$ denote the space of complex homogeneous polynomials of degree $d$ in the~$n$ variables $X_1,\ldots,X_n$. Similar to \cref{1} we call $(v,\lambda)\in(\HC^n\backslash\set{0})\times \HC$ an eigenpair of $f\in(\cH_d)^n$, if it satisfies the equation $f(v)=\lambda v$. The equivalence relation on the space of eigenpairs is defined as above: We call the pairs $(v,\lambda)$ and $(tv,t^{d-1}\lambda)$ equivalent. 

There is a canonical surjective map $(\HC^{n})^{\otimes(d+1)}\to (\cH_d)^n$, called contraction map, that maps $A$ to $f_A(X):=AX^d=\hat{A}(X,\ldots,X)$, where $X=(X_1,\ldots,X_n)$. It is easy to see that 
\begin{equation}\label{1.1}
Av^d=\lambda v, \text{ if and only if } f_A(v)=\lambda v.
\end{equation}
Let us denote by $\cH_d^\HR$ the real points in $\cH_d$, and for $f\in\cH_d^\HR$ we denote by $\#_\HR(f)$ the number of equivalence classes of real eigenpairs of $f$. If the coefficients of $f\in (\cH_d^\HR)^n$ in the \emph{Bombieri-Weyl basis} (see \cref{se:geo-framework}) are centered gaussian random variables with variance $\sigma^2=1$, we refer to $f$ as standard gaussian and write $f\sim N((\cH_d^\HR)^n)$, compare \cref{sec:random}. 

In \cref{additional_lemma} we prove that, if $A\in \left(\HR^{n}\right)^{\otimes (d+1)}$ is gaussian, then $f_A\sim N((\cH_d^\HR)^n)$. Moreover, due to~\cref{1.1} we have $\#_\HR(A)=\#_\HR(f_A)$. This implies
\begin{equation}\label{1.2}
E_{n,d}:=\mean\limits_{A \in \left(\HR^{n}\right)^{\otimes (d+1)} \text{ gaussian}} \,\#_\HR(A)\;=\mean\limits_{f\sim N((\cH_d^\HR)^n)} \,\#_\HR(f).
\end{equation}
Our main results are summarized in the following theorem.
\begin{thm}\label{main_thm}
Recall that we have defined $E_{n,d}:=\mean\limits_{A \in \left(\HR^{n}\right)^{\otimes (d+1)} \text{ gaussian}} \,\#_\HR(A)$.
\begin{enumerate}
\item We have $E_{1,d}=1$. For $n>1$ we can describe the expectation $E_{n,d}$ in terms of the Gauss hypergeometric function $~_2F_1(a,b;c;z)$ (see \cref{sec:functions}, eq. \cref{def_hypergeom}) as
	\[E_{n,d}=\frac{2^{n-1}\sqrt{d}^{\,n}\;\Gamma(n-\frac{1}{2})}{\sqrt{\pi}(d+1)^{\,n-\frac{1}{2}}\Gamma(n)}\; \left[2(n-1)~_2F_1\left(1,n-\frac{1}{2};\frac{3}{2};\frac{d-1}{d+1}\right)
	+~_2F_1\left(1,n-\frac{1}{2};\frac{n+1}{2};\frac{1}{d+1}\right)\right].\]
\item We can write $E_{n,d}$ as follows: If $n=2k>1$ is even, we have
\begin{equation*}
 E_{n,d}=\frac{1}{\sqrt{\pi}} \frac{\Gamma(n-\frac{1}{2})}{\,\Gamma(n-1)} \left (\frac{\sqrt{d}^{\,n}}{\sqrt{d+1}}\; \sum_{j=0}^{n-2} \binom{n-2}{j} \frac{ \left(-\frac{d-1}{d+1}\right)^j}{j+\frac{1}{2}}
 + 2^{n-2}\sum_{j=0}^{k-1} (-1)^j\binom{k-1}{j}  \frac{\left(\frac{1}{d+1}\right)^{j+k-\frac{1}{2}}}{j+k-\frac{1}{2}} \right).
 \end{equation*}
whereas, if $n=2k+1>1$ is odd we have 
\begin{equation*}
E_{n,d}=\frac{1}{\sqrt{\pi}} \frac{\Gamma(n-\frac{1}{2})}{\,\Gamma(n-1)} \left (\frac{\sqrt{d}^{\,n}}{\sqrt{d+1}} \sum_{j=0}^{n-2} \binom{n-2}{j} \frac{ \left(-\frac{d-1}{d+1}\right)^j}{j+\frac{1}{2}}
 + 2^{n-2}\sum_{j=0}^{k-1} (-1)^j\binom{k-1}{j}  \frac{1-\left(\frac{d}{d+1}\right)^{j+k+\frac{1}{2}}}{j+k+\frac{1}{2}} \right).
 \end{equation*}
\end{enumerate} 
Recall that we have denoted number of equivalence classes of complex eigenpairs of some generic complex tensor by $D(n,d)=\sum_{i=0}^{n-1} d^{i}$.
\begin{enumerate}[resume]
\item The asymptotic behaviour of the $E_{n,d}$ for fixed $d$ and large $n$ is 
	\[\frac{E_{n,d}}{\sqrt{D(n,d)}} \stackrel{n\to\infty}{\longrightarrow} \begin{cases} \sqrt{\frac{2}{\pi}},&\text{if } d=1. \\ 1,&\text{if } d>1.\end{cases}\]	
\item The asymptotic behaviour of the $E_{n,d}$ for fixed $n$ and large $d$ is 
	\[\frac{E_{n,d}}{\sqrt{D(n,d)}} \stackrel{d\to\infty}{\longrightarrow} 1,\quad n>1.\]	
\item The generating function of the $E_{n,d}$ for fixed $d$ is 
	\[\sum_{n=1}^\infty E_{n,d} \;z^n = \frac{z\,\left(1-z\sqrt{d}  +z\sqrt{d-2z\sqrt{d} +1}\,\right)}{(1-z^2)(1-z\sqrt{d})},\quad \text{if }\norm{z}< \frac{1}{\sqrt{d}}.\]
\end{enumerate}
\end{thm}
\begin{rem}
\cref{main_thm} generalizes the result in \cite{real_eigenvalues} from matrices $(d=1)$ to tensors. More specifically, for the matrix case the first assertion is in \cite[Section 5]{real_eigenvalues}, the third assertion is \cite[Corollary 5.2]{real_eigenvalues} and the fifth assertion is \cite[Theorem 5.1]{real_eigenvalues}.
\end{rem}
\begin{figure}[t]
  \includegraphics[width=0.8\textwidth]{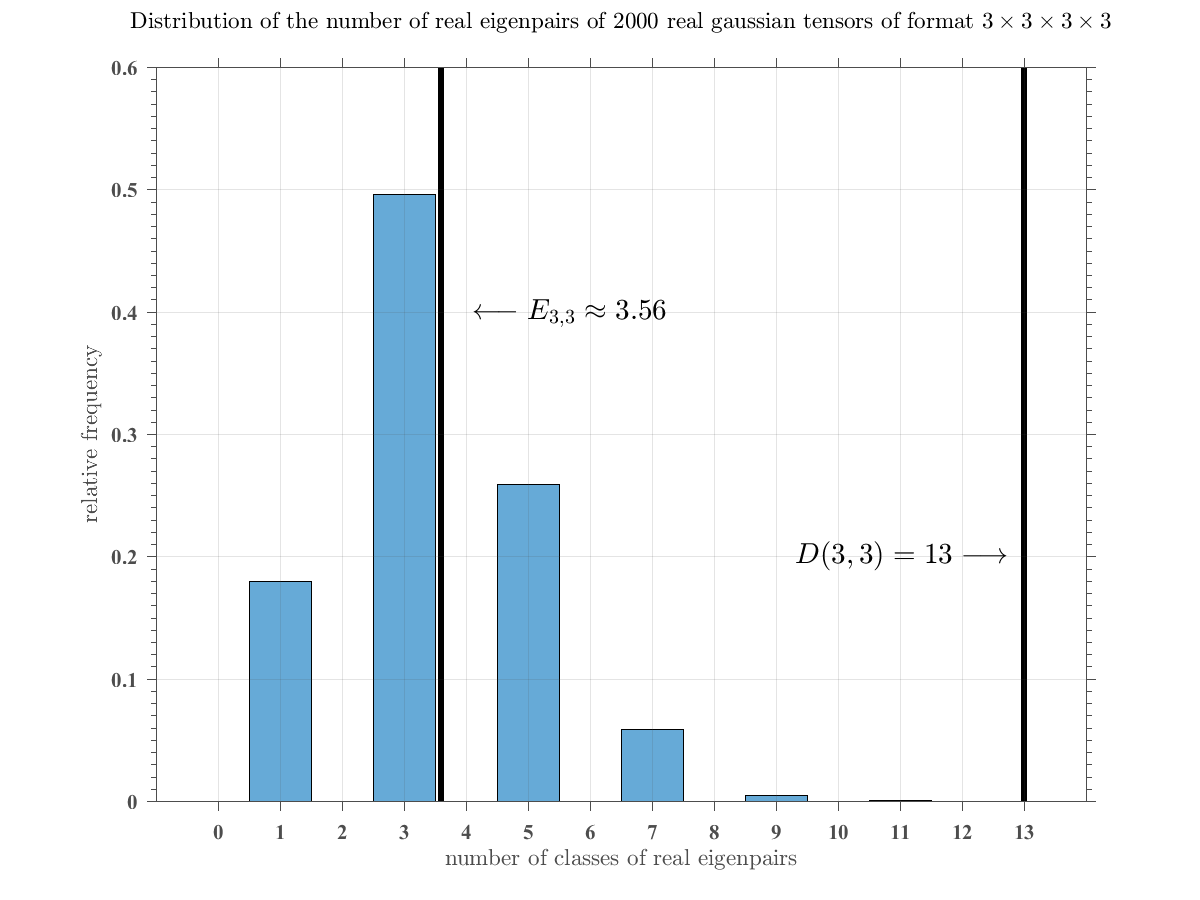}
	\caption{\small Using the \textsc{normrnd()} function in \textsc{matlab r2015b} \cite{matlab} we generated a sample of $2000$ real gaussian tensors in $\HR^{3^4}$. We used \textsc{bertini} \cite{bertini} to compute the number $\#_\HR(A)$ for each tensor. The histogram shows the relative frequencies of the $\#_\HR(A)$. The two vertical black lines represent $E_{3,3}\approx 3.56$ and $D(3,3)=13$. The reason why all the $\#_\HR(A)$ are odd numbers is that each complex eigenpair comes in a pair of conjugates, so that $\#_\HR(A) \equiv D(3,3)\equiv 1 \bmod 2$.}
\label{fig1}
\end{figure}
\subsection{Organization} 
The organization of the paper is as follows. In the next section we gather various definitions and cite theorems that we will need to prove the main theorem. In the third section we establish the geometric framework for the problem, similar to how we did in \cite[Sec. 3]{distr}. Finally, in \cref{sec:proof} we prove \cref{main_thm} and in section five we prove the lemma that implies equation \cref{1.2}.
\subsection{Acknowledgements}
The basis of this work was laid during the reunion event of the program \emph{Algorithms and Complexity in Algebraic Geometry} at the Simons Institute for the Theory of Computing. We are grateful for the Simons Institute for the stimulating environment and the financial support. We want to thank Mike Shub for pointing out to us the question about the expected number of real eigenvalues, which we answer in this paper. We further want to emphasize that this work would not have been possible without the help and support of Peter Bürgisser. The author is very thankful for all the discussions and his advice. Finally, we want to thank an anonymous referee for the detailed comments.
\section{Preliminaries}
\subsection{Differential geometry}
We denote by $\langle x,y\rangle := x^Ty$ the standard inner product on $\HR^{n}$. 
Furthermore, we set $\lVert x\rVert := \sqrt{\langle x,x\rangle}$ and 
$\HS(\HR^n):=\big\{x\in\HR^n \mid \lVert x\rVert =1\big\}$. 
We denote by
$T_x := \big\{y\in \HR^n \mid \langle x,y\rangle =0 \big\}$ 
the orthogonal complement of $x$ in $\HR^n$. If $M$ is a differentiable manifold and $x\in M$ we denote by
$T_x M$ the tangent space of $M$ at $x$. Observe that $T_x=T_x\HS(\HR^n)$.
\subsection{A general integral formula}\label{integration_formula_sec} This subsection is a summary of \cite[Sec. 13.2]{BSS}.

Let $M,N$ be Riemannian manifolds and assume that $\int_M 1 \;\d M < \infty$. Let $V\subset M\times N$ be a submanifold and assume further that $\dim V= \dim M$. Let $\pi_1: V\to M$, $\pi_2: V\to N$ be the projections onto the first and second coordinate, respectively. 

Suppose that $\pi_2$ is regular, that is every $z\in N$ is a regular value of $\pi_2$. We assume that for all regular values $x\in M$ of $\pi_1$ the fiber $\pi^{-1}(x)$ is finite. By the implicit function theorem, there exists a \emph{solution map} $S_{(x,z)}:W\to V$, defined on a neighborhood $W$ of $x$, such that $\pi_1\circ S_{(x,z)}=\mathrm{id}_W$.  Furthermore, suppose that~$\Sigma'$, the critical set of~$\pi_1$, satisfies $\dim\Sigma' < \dim V$, so that integrating over $\pi_1(V)$ we may ignore $\Sigma:=\pi_1(\Sigma')$.

For any open subset $W\subset V$ we have under the above assumptions 
	\begin{equation}\label{integration_formula}
	\int_{x\in\pi_1(U)} \lvert \pi^{-1}(x)\rvert \;\d M = \int_{z\in N} \left[ \int_{(x,z)\in \pi_1^{-1}(x)\cap U} \;\sqrt{\det\left(DS_{(x,z)}(x)\,DS_{(x,z)}(x)^T\right)} \;\;\d \pi_1^{-1}(x) \right] \d N,
	\end{equation}
provided the integrals are finite. Observe that $DS_{(x,z)}(x)$ is the linear map whose graph defines the tangent space of $V$ at $(x,z)$; in other words, we have $T_{(x,z)}V = \cset{(\dot{x},\dot{z})\in T_x M \times T_z N}{\dot{z} = DS(x)\dot{x}}$.
\begin{rem}
In \cite[Sec. 13.2]{BSS} the manifolds $M,N$ are assumed to be compact. But the deduction of formula \cref{integration_formula} can easily be extended to general $M,N$ provided $\int_M 1 \;\d M < \infty$ and $\int_{x\in\pi_1(U)} \lvert \pi^{-1}(x)\rvert \;\d M<\infty $.
\end{rem}
\subsection{The incomplete gamma, the incomplete beta and the Gauss hypergeometric function}
\label{sec:functions}
Recall the following definitions of functions: The Gamma function is defined by 
$\Gamma(n):= \int_{t=0}^\infty t^{n-1} e^{-t} \d t$ for a positive
real number $n >0$. If $n$ is a positive integer, then $\Gamma(n)=(n-1)!$. 
The \emph{upper and lower incomplete Gamma function} are denoted
\begin{equation}\label{gamma_fct}
\Gamma(n,x):= \int_{t=x}^\infty t^{n-1} e^{-t} \d t,\hspace{1cm} 
\gamma(n,x):= \int_{t=0}^x t^{n-1} e^{-t} \d t,
\end{equation}
where $x\geq 0$. The \emph{incomplete Beta function} is defined as
	\begin{equation}\label{beta_fct}
B(p,q,x):=\int_0^x t^{p-1}(1-t)^{q-1} \d t, 
\end{equation}
where $0\leq x \leq 1$ and $p,q>0$. For $a,b,c\in\HR$, $c\neq 0$, the \emph{Gauss hypergeometric function} is defined as
\begin{equation}\label{def_hypergeom}
~_2F_1(a,b;c; x) := \sum_{k=0}^\infty \frac{(a)_k\,(b)_k }{(c)_k}\, \frac{x^ k}{k!},
\end{equation}
where $(z)_k:=z(z+1)\cdot\ldots\cdot(z+k-1)$ is the Pochhammer polynomial. 
The following is \cite[eq. 6.455]{gradshteyn}.
\begin{prop}\label{fancy_integrals}
For $\alpha,\beta,\mu,\nu>0$ we have
\begin{enumerate}
\item
\[\int_{0}^\infty x^{\mu-1} e^{-\beta x} \Gamma(\nu,\alpha x) \d x = \frac{\alpha^\nu \Gamma(\mu+\nu)}{\mu(\alpha+\beta)^{\mu+\nu}}\; ~_2F_1\left(1,\mu+\nu; \mu+1; \frac{\beta}{\alpha+\beta}\right).\]
\item 
\[\int_{0}^\infty x^{\mu-1} e^{-\beta x} \gamma(\nu,\alpha x) \d x = \frac{\alpha^\nu \Gamma(\mu+\nu)}{\nu(\alpha+\beta)^{\mu+\nu}}\; ~_2F_1\left(1,\mu+\nu; \nu+1; \frac{\alpha}{\alpha+\beta}\right).\]
\end{enumerate}
\end{prop}
The following proposition is a combination of \cite[eq. (3.9), (3.10)]{temme} and \cite[equation after (3.14)]{temme}.
\begin{prop}\label{beta_asymptotic}
Let $\mathrm{erfc}(\cdot)$ denote the complementary error function. For $p>0$, $0\leq x \leq 1$ and~$q\to\infty$ we have
	\[B(p,q,x)\sim \frac{\Gamma(p)\,\Gamma(q)}{\Gamma(p+q)} \left( \frac{1}{2} \;\mathrm{erfc}\left(-\omega\right) + b_x(p,q)\right),\]
where
	\[\omega := \left[p\ln \left(\frac{p}{p+q}\right) - p\ln(x) + q\ln\left(\frac{q}{p+q}\right) -q\ln(1-x)\right]^\frac{1}{2}\]
and
\[b_x(p,q):= \left(\frac{p}{2\pi q(p+q)}\right)^\frac{1}{2}\left(\frac{x(p+q)}{p}\right)^p \left(\frac{(1-x)(p+q)}{q}\right)^q \left(\frac{q}{p-(p+q)x} + \omega^{-1} \sqrt{\frac{p+q}{p}}\right) (1+\cO(q^{-1})).\]
\end{prop}
Next we give some relations between the Gauss hypergeometric function and the beta function.
\begin{prop}\label{hypergeom_as_beta}
Let $b,c,\nu\in\HR$ and $0< x <1$. 
\begin{enumerate}
\item If $c-1>0, b-c+1>0$, then $~_2F_1(1,b;c;x)= (c-1) (1-x)^{c-b -1} x^{1-c} B(c-1,b-c+1,x)$.
\item If $m$ is a positive integer, then 
\[B(\nu,m,x)=x^\nu \sum_{j=0}^{m-1} \binom{m-1}{j}  \frac{(-x)^j}{j+\nu}.\]
\item If $m,n$ are non-negative integers, then 
\[B\left(n+\frac{1}{2},m+1,x\right)=x^{n+\frac{1}{2}} \sum_{j=0}^{m} \binom{m}{j}  \frac{(-x)^j}{j+n+\frac{1}{2}}\] 
and 
\[B\left(n+1,m+\frac{1}{2},x\right)=\sum_{j=0}^{n} (-1)^j \binom{n}{j} \frac{1-(1-x)^{j+m+\frac{1}{2}}}{j+m+\frac{1}{2}}.\]
\item For all $n$ we have
	\[~_2F_1\left(1,n-\frac{1}{2};\frac{3}{2};x\right)= \frac{1}{2(1-x)^{n-1}} \sum_{j=0}^{n-2} \binom{n-2}{j} \frac{(-x)^j}{j+\frac{1}{2}}.\]
\item If $n=2k$ is even, then
	\[~_2F_1\left(1,n-\frac{1}{2};\frac{n+1}{2};x\right)= \frac{n-1}{2(1-x)^{k}}  \sum_{j=0}^{k-1} \binom{k-1}{j}  \frac{(-x)^{j}}{j+k-\frac{1}{2}}.\]
\item If $n=2k+1$ is odd, then
	\[~_2F_1\left(1,n-\frac{1}{2};\frac{n+1}{2};x\right)= \frac{n-1}{2(1-x)^{k+\frac{1}{2}} x^k}\sum_{j=0}^{k-1} (-1)^j \binom{k-1}{j} \frac{1-(1-x)^{j+k+\frac{1}{2}}}{j+k+\frac{1}{2}} \]
\end{enumerate}
\end{prop}
\begin{proof}
Part (1) can be found in the first table in \cite[Sec. 60:4]{atlas}, Part (2) is \cite[58:4:6]{atlas} and Part (3) is \cite[58:4:7 and 58:4:8]{atlas}. Finally, (4) is a combination of (1) and (2) applied to the special case $b=n-\frac{1}{2}, c=\frac{3}{2}$ and (5)--(6) follow from (1) and (3) with $b=n-\frac{1}{2}, c=\frac{n+1}{2}$.
\end{proof}
\subsection{The expected absolute value of the characteristic polynomial of a gaussian matrix}\label{sec:random}
We say that a random variable $z\in\HR$ is \emph{centered normal with variance $\sigma^2$} if $z$ is distributed with the density function $\frac{1}{\sqrt{2\pi}}\cdot \exp\big(-\tfrac{z^2}{2\sigma^2}\big),$ and we write $z\sim N(0,\sigma^2)$ for this distribution. The following lemma is well-known:
\begin{lemma}\label{sum_of_normal}
\begin{enumerate}
\item Suppose that $z\sim N(0,\sigma^2)$ and $t\in\HR\backslash\set{0}$. Then $tz\sim N(0,t^2\sigma^2)$.
\item If $z_i\sim N(0,\sigma_i^2), i=1,2$, are independent, then $z_1+z_2\sim N(0,\sigma_1^2+\sigma_2^2)$.
\end{enumerate}
\end{lemma}
We say that $z$ is \emph{standard normal}, if $\sigma^2=1$. More generally, if $E$ is a finite dimensional real vector space with inner product, we define the standard normal density on the space $E$ as 	
	\begin{equation*}\label{varphi_E}
	\varphi_E(z) :=   \frac{1}{\sqrt{2\pi}^{\,\dim (E)}}\cdot \exp\left(-\frac{\lVert z\rVert^2}{2}\right).
	\end{equation*}
If it is clear from the context which space is meant, we sometimes omit the subscript~$E$ in $\varphi_E$. If $z\in E$ is a random variable with density $\varphi_E$, we write $z\sim N(E)$ (compare the notion $N((\cH_d^\HR)^n)$ from the introduction).
\begin{thm}\label{expectation_det}
Let $I_{n}$ denote the $n\times n$-identity matrix. We have for standard gaussian $A\in\HR^{n\times n}$ and fixed $t\in \HR$ 
\[\mean\limits_{A\sim N(\HR^{n\times n})} \lvert \det (A+tI_{n})\rvert =  \frac{\sqrt{2}^{\,n}}{\sqrt{\pi}}\;\frac{\Gamma\left(\frac{n+1}{2}\right)}{\Gamma(n)}\;  \left[e^{\frac{t^2}{2}}\; \Gamma(n,t^2) + 2^{n-1}\left(\frac{t^2}{2}\right)^{\frac{n}{2}} \gamma\left(\frac{n}{2},\frac{t^2}{2}\right)\right],\]
where $\Gamma(n,z)$ and $\gamma(n,z)$ are the upper and lower incomplete gamma function, respectively.
\end{thm}
\begin{proof}
Put $B:=(A+tI_n)^T(A+tI_n)$. Then $B$ is said to have the \emph{noncentral Wishart distribution}; see \cite[Definition 10.3.1, p. 441]{muirhead}. We have $\lvert \det(A+tI_n)\rvert=\det(B)^\frac{1}{2}$ and the expectation of $\det(B)^\frac{1}{2}$ is given in \cite[Theorem 10.3.7, p. 447]{muirhead}. Combining this with \cite[Theorem 4.1]{real_eigenvalues} yields the claim.
\end{proof}

~
\section{Geometric framework}\label{se:geo-framework} 
As in the introduction we denote by~$\cH_d^\HR$ the vector space of real
homogeneous polynomials of degree~$d$ in the $n$ variables
$X_1,\ldots,X_n$. For a vector of non-negative integers $\alpha=(\alpha_1,\ldots,\alpha_n)$ we denote $\norm{\alpha}:=\alpha_1+\ldots+\alpha_n$. The \emph{Bombieri-Weyl basis} on~$\cH_d^\HR$ is defined as $\cE=\big\{e_\alpha:=\sqrt{\binom{d}{\alpha}} \;X^\alpha \mid \lvert \alpha \rvert = d\big\}$, where $\binom{d}{\alpha}=\frac{d!}{\alpha_1!\cdots\alpha_n!}$ is the multinomial coeffient. The Bombieri-Weyl product on $\cH_d^\HR$ is defined as~$
	\Big\langle \sum a_\alpha e_\alpha,\sum b_\alpha e_\alpha \Big\rangle := \sum\limits_\alpha a_\alpha b_\alpha.$ This product extends to $(\cH_d^\HR)^n$ as follows. Let $f=(f_1,\ldots,f_n)$ and $g=(g_1,\ldots,g_n)\in(\cH_d^\HR)^n$. We define $\langle f,g\rangle :=\sum_{i=1}^n \langle f_i,g_i\rangle.$ For $f\in(\cH_d^\HR)^n$ we set $\lVert f\rVert := \sqrt{\langle f,f\rangle}$. The orthogonal group $\cO(n)$ acts on $\cH_d^\HR$ via $U.f:=f\circ U^{-1}$. The following is \cite[Theorem 16.3]{condition}.
\begin{thm}\label{thm_unitary}
For all $f,g\in\cH_d^\HR$ and $U\in\cO(n)$ we have $\langle U.f,U.g\rangle = \langle f,g\rangle$.
\end{thm}
Let $e_1:=(1,0,\ldots,0)\in\HR^n$ and consider the space $R:=\cset{f\in \cH_d^\HR}{f(e_1)=0, Df(e_1)=0}$ of polynomials that vanish at first order at $e_1$. Let $L:=R^\perp \cap \set{f(e_1)=0}$ and $C:=R^\perp \cap L^\perp$. The following is \cite[Prop. 16.16]{condition}.
\begin{prop}\label{decomp}
Denote $X:=(X_1,\ldots,X_n)$ and $X':=(X_2,\ldots,X_n)$ 
\begin{enumerate}
\item $\cH_d^\HR = C \oplus L \oplus R$ is an orthogonal decomposition 
\item $C=\cset{c X_1^d}{c\in \HR}$ and $\Norm{c X_1^d}=\norm{c}$.
\item $L=\cset{\sqrt{d} X_1^{d-1} a^T X'}{a\in\HR^{n-1}}$ and $\Norm{\sqrt{d} X_1^{d-1} a^T X'}=\norm{a}$.
\end{enumerate}
\end{prop}
\subsection{The solution manifold}
Let $\cA:=\HR[X_1,\ldots,X_n,\Lambda]$ be the polynomial ring in the variables $X_1,\ldots,X_n,\Lambda$ and define $F: (\cH_d^\HR)^n \to \cA^n,f\mapsto f(X)-\Lambda\,  X.$
For $f\in (\cH_d^\HR)^n$ we set $F_f:=F(f)$, such that 
\begin{equation}\label{dfn_F_A}
F_f:\HR^n\times \HR \to \HR^n,\quad (v,\lambda) \mapsto F_f(v,\lambda)= f(v)-\lambda v.
\end{equation}
We denote by $\partial_X$ and $\partial_\Lambda$ the partial derivatives with respect to $X=(X_1,\ldots,X_n)$ and $\Lambda$, respectively. Let $I_{n}$ denote the $n\times n$-identity matrix. Then the derivative of $F_f$ at $(v,\lambda)$ has the following matrix representation:
\begin{equation}
	\begin{bmatrix} \partial_X f -\partial_X (\Lambda \, X), &  -\partial_\Lambda(\Lambda\, X)\end{bmatrix}_{(X,\Lambda)=(v,\lambda)}
	= \begin{bmatrix} \partial_X f(v) -\lambda I_{n}, &  -v\end{bmatrix}\label{jacobian},
\end{equation}  
As in \cite[Sec. 3]{distr} we define 
\begin{align*}\label{sol_variety}
	\hV &:=\cset{(f,v,\lambda)\in(\cH_d^\HR)^n\times \HS(\HR^n)\times\HR}{f(v)=\lambda v}, \quad 
	\hW :=\cset{(f,v,\lambda)\in \hV}{\rank DF_f(v,\lambda)=n},
\end{align*}
the \emph{real solution manifold} and its subset, the \emph{manifold of real well-posed triples}.  Moreover, we define the projections 
\begin{equation}
\label{proj}\pi_1\colon\hV\to (\cH_d^\HR)^n,\; (f,v,\lambda)\mapsto f,\quad \pi_2\colon\hV\to \HS(\HR^n)\times \HR,\; (f,v,\lambda)\mapsto (v,\lambda).
\end{equation}
The orthogonal group $\cO(n)$ acts on 
$\hV$ via  
\begin{equation}\label{groupaction}
U.(f,v,\lambda) := (U\circ f\circ U^{-1}, Uv, \lambda),\quad U\in\cO(n).
\end{equation}
Note that $\hW$ is invariant under this group action and that $\cO(n)$ acts by isometries by \cref{thm_unitary}. 
\begin{lemma}\label{tangent_space}
For $v\in \HR^n$ we denote $\mathrm{eval}_v :(\cH_d^\HR)^n\to \HR^n, f\mapsto f(v)$. 
We have that $\hV$ is a Riemannian manifold of dimension $\dim (\cH_d^\HR)^n$. The tangent space of $\hV$ at~$(f,v,\lambda)\in\hW$ is given by 
	\[\cset{(\dot{f},\dot{v},\dot{\lambda})\in (\cH_d^\HR)^n\times T_v\times \HR}{ (\dot{v},\dot{\lambda}) 
 = - DF_f(v,\lambda)|_{T_v\times \HR}^{-1}\; \mathrm{eval}_v(\dot{f})}.\]
\end{lemma}
\begin{proof}
The set $\hV$ is the zero set of the $\cC^1$-function $G(f,v,\lambda)=F_f(v,\lambda)=f(v)-\lambda v$. The derivative of~$G$ at $(f,v,\lambda)$ maps $(\dot{f},\dot{v},\dot{\lambda}) \mapsto \dot{f}(v)+DF_f(v,\lambda)(\dot{v},\dot{\lambda})$. Setting $(\dot{v},\dot{\lambda})=0$ we see that the image of~$DG(f,v,\lambda)$ contains $\cset{\dot{f}(v)}{\dot{f}\in(\cH_d^\HR)^n}=\HR^n$, which shows that $G$ is a submersion. These two facts combined with \cite[Theorem A.9]{condition} imply that $\hV$ is a Riemannian manifold, whose tangent space at $(f,v,\lambda)$ is $\cset{(\dot{f},\dot{v},\dot{\lambda})\in(\cH_d^\HR)^n\times T_v\times \HR}{ \dot{f}(v) 
       + DF_f(v,\lambda)(\dot{v},\dot{\lambda})=0}$. Moreover, the dimension of $\hV$ is given by $\dim \hV = \dim \, (\cH_d^\HR)^n \times \HS(\HR^n)\times \HR - \dim \HR^n = \dim \, (\cH_d^\HR)^n$. Clearly, $\dim \hV= \dim \hW$. Let $(f,v,\lambda)\in\hW$, i.e. $\rank DF_f(v,\lambda)=n$. From \cref{jacobian} and Euler's identity for homogeneous functions it follows that for $(f,v,\lambda)\in\hW$ we have $\ker DF_f(v,\lambda)=(v,(d-1)\lambda)^T\,\HR$. In particular, $DF_f(v,\lambda)|_{T_v\times \HR}$ is invertible, so $T_{(f,v,\lambda)}\hV$ can be written in the desired form.
\end{proof}
\section{Proof of the main theorem}\label{sec:proof}
As in the introduction we denote by $\#_\HR(f)$ the number of real eigenvalues of $f$. Recall from \cref{1.2} that 
\[E_{n,d}=\mean\limits_{f\sim N((\cH_d^\HR)^n)} \,\#_\HR(f).\]
Moreover, recall that $\Gamma(n,x)$ and $\gamma(n,x)$ are the upper and lower incomplete gamma functions, see \cref{gamma_fct}. For $\lambda\in\HR$ we put $F_{0,d}(\lambda)=1$ and define
	\begin{equation}\label{F_nd}
	F_{n,d}(\lambda):= \frac{\sqrt{d}^{\,n}}{\Gamma(n)} 
         \left[ e^{\lambda^2/(2d)}\;\Gamma\left(n,\frac{\lambda^2}{d}\right) + 2^{n-1}\left(\frac{\lambda^2}{2d}\right)^\frac{n}{2} \gamma\left(\frac{n}{2},\frac{\lambda^2}{2d}\right)\right],\; n\geq 1.
	\end{equation}
The main step on the way to prove \cref{main_thm} is the following proposition.
\begin{prop}\label{main_step}
We have $\mean\limits_{f\sim N((\cH_d^\HR)^n)} \,\#_\HR(f) = \mean\limits_{\lambda\sim N(0,1)} F_{n-1,d}(\lambda)$. 
\end{prop}
For $n=1$ we have $\#_\HR(f)=1$, so in this case \cref{main_step} is immediate. For the case $n>1$ a proof follows from combining \cref{lemma1} and \cref{shape_E} below. 

For the rest of this subsection let $n>1$. Recall from \cref{proj} the definition of $\pi_1:\hV\to (\cH_d^\HR)^n$ and from the introduction the definition of $D(n,d):=\sum_{i=0}^{n-1}d^{i}$. For $f\in(\cH_d^\HR)^n$ we have $\#_\HR(f)=\frac{1}{2} \lvert \pi_1^{-1}(f)\rvert$ and by \cite[Cor. 1]{qi1} and \cite[Theorem 1.2]{sturmfels-cartwright} we have~$\#_\HR(f) \leq D(n,d)$. In particular, if $f$ is a regular value of $\pi_1$, the fiber~$\pi_1^{-1}(f)$ is finite and \[\int_{(\cH_d^\HR)^n}  \lvert\pi_1^{-1}(f)\rvert\; \d (\cH_d^\HR)^n \leq  \int_{(\cH_d^\HR)^n}  2D(n,d) \;\d (\cH_d^\HR)^n = 2D(n,d) <\infty.\]
Let $\mu$ denote the Lebesgue measure on $(\cH_d^\HR)^n$ and put $M=(\cH_d^\HR)^n$, $\d M= \varphi_{(\cH_d^\HR)^n}(f)\, \d \mu$, $N=\HS(\HR^n)\times \HR$, $V=\hV$ and $\Sigma' = \hV\backslash\hW$. We are now in the situation of the integral formula \cref{integration_formula}.  As in \cref{integration_formula_sec} we denote by~$S_{(f,v,\lambda)}$ the (locally defined) solution map. \cref{tangent_space} shows that~$DS_{(f,v,\lambda)}(\dot{f})=\restr{DF_f(v,\lambda)}{T_v\times\HR}^{-1}\, \circ \,\mathrm{eval}_v$. It is elementary to prove $\mathrm{eval}_v\mathrm{eval}_v^T = \id_{\HR^n}$. Hence, by the integral formula we have that $\mean\limits_{f\sim N((\cH_d^\HR)^n)} \,\#_\HR(f)$ equals
\begin{equation}\label{eq1}
 \frac{1}{2} \int_{(v,\lambda)\in\HS(\HR^n)\times \HR} \left(\;\int_{(f,v,\lambda)\in\pi_2^{-1}(v,\lambda)} \left\lvert \det DF_f(v,\lambda)|_{T_{v}\times\HR} \right\rvert  \,\varphi(f) \,\d \pi_2^{-1}(v,\lambda)\right) \,\d (\HS(\HR^n)\times \HR);
\end{equation}
Let us put
	\begin{equation}\label{dfn_E}
	J(v,\lambda):=\int_{(f,v,\lambda)\in\pi_2^{-1}(v,\lambda)} \left\lvert\det DF_f(v,\lambda)|_{T_{v}\times\HR} \right\rvert \,\varphi(f) \,\d \pi_2^{-1}(v,\lambda).
	\end{equation}
Using the group action of $\cO(n)$ on $\hV$ from \cref{groupaction} we obtain $J(v,\lambda)=J(e_1,\lambda)$, where $e_1:=(1,0,\ldots,0)$.
\begin{lemma}\label{lemma1}
Let $(v,\lambda)\in \HS(\HR^n)\times\HR$. We have $\mean\limits_{f\sim N((\cH_d^\HR)^n)} \,\#_\HR(f)  =\frac{\sqrt{\pi}^{\,n}}{\Gamma(\frac{n}{2})} \int_{\lambda \in \HR} J(e_1,\lambda)\; \d \HR.$
\end{lemma}
\begin{proof}
From \cref{eq1} and \cref{dfn_E} we have $\mean\limits_{f\sim N((\cH_d^\HR)^n)} \,\#_\HR(f)  =\frac{1}{2} \int_{(v,\lambda)\in \HS(\HR^n)\times \HR} \;J(v,\lambda) \,\d (\HS(\HR^n)\times\HR)$. In this integral we may replace $J(v,\lambda)$ by $J(e_1,\lambda)$. Since $J(e_1,\lambda)$ is independent of $v$, we can integrate over $v$ to obtain
	\[\mean\limits_{f\sim N((\cH_d^\HR)^n)} \,\#_\HR(f)  = \frac{\mathrm{vol}\,\HS(\HR^n)}{2} \int_{\lambda \in \HR} J(e_1,\lambda)\; \d \HR.\]
The claim follows from the fact that $\mathrm{vol}\,\HS(\HR^n) =\frac{2\sqrt{\pi}^{\,n}}{\Gamma(\frac{n}{2})}$.
\end{proof}
\begin{prop} \label{shape_E}
For any $\lambda\in\HR$ we have
\[J(e_1,\lambda)=\frac{\sqrt{d}^{\,n-1}\, \Gamma(\frac{n}{2})}{\sqrt{\pi}^{\,n}} 
         \left( e^{\lambda^2/(2d)}\; \frac{\Gamma\left(n-1,\frac{\lambda^2}{d}\right)}{\Gamma(n-1)} + 2^{n-2}\left(\frac{\lambda^2}{2d}\right)^\frac{n-1}{2} \frac{\gamma\left(\frac{n-1}{2},\frac{\lambda^2}{2d}\right)}{\Gamma(n-1)}\right)\;\varphi(\lambda),\]
where $\varphi$ denotes the density function of the standard normal distribution.
\end{prop}
\begin{proof}
The following is very similar to the proof of \cite[Proposition 4.1]{distr}. 

We have $\pi_2^{-1}(e_1,\lambda) =V(e_1,\lambda)\times  \set{e_1}\times\set{\lambda}$, where  
\[V(e_1,\lambda) :=  X_1^d \lambda e_1 + \Big\{g\in(\cH_d^\HR)^n \mid g(e_1)=0 \Big\}.\]
Observe that $\pi_2^{-1}(e_1,\lambda)$ is isometric to $V(e_1,\lambda)$. So in the description of $J(e_1,\lambda)$ \cref{dfn_E} we integrate over~$V(e_1,\lambda)$:	\begin{equation}\label{pushforward_integral2}
	J(e_1,\lambda)=  \int_{f\in V(e_1,\lambda)}  \left\lvert \det DF_f(e_1,\lambda)|_{T_{e_1}\times\HR} \right\rvert  \varphi_{V(e_1,\lambda)}(f) \;\d V(e_1,\lambda)
	\end{equation}
Let $R:=\Big\{h\in \cH_d^\HR \mid h(e_1)=0, Dh(e_1)=0 \Big\}$. 
By \cref{decomp}, for any $f\in V(e_1,\lambda)$, there exist uniquely determined $h\in R^n$ and $M\in \HR^{n \times (n-1)}$ 
such that we can orthogonally decompose $f$ as
	\begin{equation}\label{aa2}
	f = X_1^d \lambda e_1 + X_1^{d-1}\sqrt{d}\, M \, (X_2,\ldots,X_n)^T\, +h.
	\end{equation}
The matrix representation of $DF_f(X,\Lambda)$ from \cref{jacobian} can the be written as
	\begin{equation*}
	DF_f(e_1,\lambda)=\begin{bmatrix} \partial_X f(e_1,\lambda) - \lambda I_{n},&
          -e_1\end{bmatrix} = \begin{bmatrix} (d-1)\lambda &
          \sqrt{d}\cdot a & -1 \\ 0 & \sqrt{d} A-\lambda I_{n-1} & 0 \end{bmatrix} \in\HR^{n\times (n+1)},
	\end{equation*}
where  $a\in\HR^{1 \times (n-1)}$ is the first row of $M$ and $A\in
\HR^{(n-1) \times (n-1)}$ is the matrix that is obtained by removing
the first row of $M$. Hence,
\begin{equation}\label{aa3}
\det DF_f(e_1,\lambda)|_{T_{e_1}\times \HR} = -\det\;(\sqrt{d} A-\lambda I_{n-1}).
\end{equation}
The summands in \cref{aa2} are pairwise orthogonal, which implies that
	\begin{align}\label{aa4}
	\varphi_{V(e_1,\lambda)}(f) &=\varphi_{\HR^n}(\lambda e_1)\cdot \varphi_{\HR^{(n-1)\times (n-1)}}(A)\cdot \varphi_{\HR^{n-1}}(a)\cdot \varphi_R(h)\\
	&= \frac{1}{\sqrt{2\pi}^{\,n-1}} \cdot\varphi(\lambda)\cdot \varphi_{\HR^{(n-1)\times (n-1)}}(A)\cdot \varphi_{\HR^{n-1}}(a)\cdot \varphi_R(h).\nonumber
	\end{align}
We plug \cref{aa3} and \cref{aa4} into \cref{pushforward_integral2}. By \cref{aa3}, $\det DF_f(e_1,\lambda)|_{T_{e_1}\times \HR}$ is independent of $a$ and $h$. We may therefore integrate over $a$ and $h$ without changing the value of the integral. From this we see that
\begin{align*}
J(e_1,\lambda)=& \frac{1}{\sqrt{2\pi}^{\,n-1}} \; \varphi(\lambda)\;
        \mean\limits_{A\sim N(\HR^{(n-1)\times(n-1)})} \left\lvert\det\left(\sqrt{d}A-\lambda I_{n-1}\right) \right\rvert\\
        = & \frac{\sqrt{d}^{\,n-1}}{\sqrt{2\pi}^{\,n-1}} \; \varphi(\lambda)\;
        \mean\limits_{A\sim N(\HR^{(n-1)\times(n-1)})} \left\lvert\det\left(A-\frac{\lambda}{\sqrt{d}}\, I_{n-1}\right) \right\rvert\\
        = &\frac{\sqrt{d}^{\,n-1}}{\sqrt{\pi}^{\,n}} \; \varphi(\lambda)\;\frac{ \Gamma(\frac{n}{2})}{\Gamma(n-1)}
         \left( e^{\lambda^2/(2d)}\; \Gamma\bigg(n-1,\frac{\lambda^2}{d}\bigg) + 2^{n-2}\left(\frac{\lambda^2}{2d}\right)^\frac{n-1}{2} \gamma\left(\frac{n-1}{2},\frac{\lambda^2}{2d}\right)\right);
	\end{align*}
the last line by \cref{expectation_det}.
\end{proof}
\vspace{-0.5cm}
\subsection{Proof of Theorem 1.1 (1)} 
The case $n=1$ is trivial. Let $n>1$. From \cref{main_step} we get that
	\begin{align*}
	E_{n,d} \; = \; &\frac{\sqrt{d}^{\,n-1}}{\sqrt{2\pi}\,\Gamma(n-1)} \;  \int_{\lambda\in\HR} e^{- \frac{\lambda^{2}}{2}(1-d^{-1})}\;\Gamma\bigg(n-1,\frac{\lambda^2}{d}\bigg)\; \d \lambda \\
&+\frac{\sqrt{2}^{\,n-3}}{\sqrt{2\pi}\;\Gamma(n-1)}\; \int_{\lambda\in\HR} e^{- \frac{\lambda^{2}}{2}}\;\norm{\lambda}^{n-1}\; \gamma\left(\frac{n-1}{2},\frac{\lambda^2}{2d}\right)\;\d \lambda.
	\end{align*}
Making the substitution of variables $x:=\lambda^2$, such that $ \d \lambda = \frac{\d x}{2\sqrt{x}}$, we obtain
	\begin{align*} \int_{\lambda\in\HR} e^{- \frac{\lambda^{2}}{2}(1-d^{-1})}\;\Gamma\bigg(n-1,\frac{\lambda^2}{d}\bigg)\; \d \lambda &= \int_{x>0} x^{-1/2} e^{-\frac{x}{2}(1-d^{-1})}\;\Gamma\bigg(n-1,\frac{x}{d}\bigg)\; \d x\\
	&= \frac{2^{n+1/2}\,\Gamma(n-\frac{1}{2})}{d^{n-1}(1+d^{-1})^{n-1/2}}\; ~_2F_1\left(1,n-\frac{1}{2};\frac{3}{2};\frac{d-1}{d+1}\right),
	\end{align*}
the last line by \cref{fancy_integrals}(1). In the same way, but using \cref{fancy_integrals}(2), we get 
	\begin{align*} 
	\int_{\lambda\in\HR} e^{-\frac{\lambda^ {2}}{2}}\;\lvert\lambda\rvert^{n-1}\; \gamma\left(\frac{n-1}{2},\frac{\lambda^2}{2d}\right)\;\d \lambda=&\int_{x>0} e^{- \frac{x}{2}}\;x^\frac{n-2}{2}\; \gamma\left(\frac{n-1}{2},\frac{x}{2d}\right)\;\d \lambda\\
	= &\frac{\sqrt{2}^{\,n+2}\;\Gamma(n-\frac{1}{2})}{\sqrt{d}^{\,n-1} (n-1) \,(1+d^{-1})^{n-\frac{1}{2}}}\,~_2F_1\left(1,n-\frac{1}{2};\frac{n+1}{2};\frac{1}{d+1}\right).
	\end{align*}
Hence,
\begin{align*}
	  E_{n,d}=&\;\frac{\sqrt{d}^{\,n-1}}{\sqrt{2\pi}\,\Gamma(n-1)} \;\frac{2^{n+1/2}\,\Gamma(n-\frac{1}{2})}{d^{n-1}(1+d^{-1})^{n-1/2}}\; ~_2F_1\left(1,n-\frac{1}{2};\frac{3}{2};\frac{d-1}{d+1}\right)\\
	&\hspace{0.5cm}+\frac{\sqrt{2}^{\,n-3}}{\sqrt{2\pi}\;\Gamma(n-1)}\;\frac{\sqrt{2}^{\,n+2}\;\Gamma(n-\frac{1}{2})}{\sqrt{d}^{\,n-1} (n-1) \,(1+d^{-1})^{n-1/2}}\,~_2F_1\left(1,n-\frac{1}{2};\frac{n+1}{2};\frac{1}{d+1}\right)\\
		 = &\;\frac{2^{n-1}\sqrt{d}^{\,n}}{\sqrt{\pi}(d+1)^{n-\frac{1}{2}}} \;\frac{\Gamma(n-\frac{1}{2})}{\,\Gamma(n)}\;\left[2(n-1)\; ~_2F_1\left(1,n-\frac{1}{2};\frac{3}{2};\frac{d-1}{d+1}\right)+  ~_2F_1\left(1,n-\frac{1}{2};\frac{n+1}{2};\frac{1}{d+1}\right)\right];
\end{align*}
for the last line we have used that $(n-1)\Gamma(n-1)=\Gamma(n)$. 
\subsection{Proof of Theorem 1.1 (2)}
Recall from \cref{main_thm} (1) that
	\[E_{n,d}=\frac{2^{n-1}\sqrt{d}^{\,n}\;\Gamma(n-\frac{1}{2})}{\sqrt{\pi}(d+1)^{\,n-\frac{1}{2}}\Gamma(n)}\; \left[2(n-1)~_2F_1\left(1,n-\frac{1}{2};\frac{3}{2};\frac{d-1}{d+1}\right)
	+~_2F_1\left(1,n-\frac{1}{2};\frac{n+1}{2};\frac{1}{d+1}\right)\right].\]
By \cref{hypergeom_as_beta}(4) we can substitute
	\begin{equation*}\label{eq50}
	~_2F_1\left(1,n-\frac{1}{2};\frac{3}{2};\frac{d-1}{d+1}\right)= \frac{(d+1)^{n-1}}{2^{n}} \sum_{j=0}^{n-2} \binom{n-2}{j} \frac{1}{j+\frac{1}{2}} \left(-\frac{d-1}{d+1}\right)^j.
	\end{equation*}
We now distinguish two cases. If $n=2k$ is even, by \cref{hypergeom_as_beta}(5) we have
	\begin{equation}\label{eq51}
	~_2F_1\left(1,n-\frac{1}{2};\frac{n+1}{2};\frac{1}{d+1}\right)= \frac{(n-1)(d+1)^{\frac{n}{2}}}{2 d^{\frac{n}{2}}}  \sum_{j=0}^{k-1} \binom{k-1}{j}  \frac{1}{j+k-\frac{1}{2}} \left(-\frac{1}{d+1}\right)^{j},
	\end{equation}
which shows that
\begin{equation*}\label{eq52}
E_{n,d}=\frac{1}{\sqrt{\pi}} \frac{\Gamma(n-\frac{1}{2})}{\,\Gamma(n-1)} \left (\frac{\sqrt{d}^{\,n}}{\sqrt{d+1}}\; \sum_{j=0}^{n-2} \binom{n-2}{j} \frac{ \left(-\frac{d-1}{d+1}\right)^j}{j+\frac{1}{2}}
 + 2^{n-2}\sum_{j=0}^{k-1} (-1)^j\binom{k-1}{j}  \frac{\left(\frac{1}{d+1}\right)^{j+k-\frac{1}{2}}}{j+k-\frac{1}{2}} \right).
 \end{equation*}
If $n=2k+1$ is odd, then by \cref{hypergeom_as_beta}(6) we have
	\begin{equation}\label{eq53}
	~_2F_1\left(1,n-\frac{1}{2};\frac{n+1}{2};\frac{1}{d+1}\right)= \frac{(n-1)(d+1)^{n-\frac{1}{2}}}{2 d^{\frac{n}{2}}} \sum_{j=0}^{k-1} (-1)^j \binom{k-1}{j} \frac{1-\left(\frac{d}{d+1}\right)^{j+k+\frac{1}{2}}}{j+k+\frac{1}{2}}
	\end{equation}
and therefore in this case $E_{n,d}$ equals
\begin{equation*}
\frac{1}{\sqrt{\pi}} \frac{\Gamma(n-\frac{1}{2})}{\,\Gamma(n-1)} \left (\frac{\sqrt{d}^{\,n}}{\sqrt{d+1}} \sum_{j=0}^{n-2} \binom{n-2}{j} \frac{ \left(-\frac{d-1}{d+1}\right)^j}{j+\frac{1}{2}}
 + 2^{n-2}\sum_{j=0}^{k-1} (-1)^j\binom{k-1}{j}  \frac{1-\left(\frac{d}{d+1}\right)^{j+k+\frac{1}{2}}}{j+k+\frac{1}{2}} \right),
 \end{equation*}
which finishes the proof.
\subsection{Proof of Theorem 1.1 (3)} The case $d=1$ was proven in \cite[Corollary 5.2]{real_eigenvalues}. Let $d>1$. Recall from \cref{main_thm} (1) that $\frac{E_{n,d}}{\sqrt{D(n,d)}}$ equals
		\[\frac{1}{\sqrt{D(n,d)}} \,\frac{2^{n-1}\sqrt{d}^{\,n}\;\Gamma(n-\frac{1}{2})}{\sqrt{\pi}(d+1)^{\,n-\frac{1}{2}}\Gamma(n)}\; \left[2(n-1)~_2F_1\left(1,n-\frac{1}{2};\frac{3}{2};\frac{d-1}{d+1}\right)
	+~_2F_1\left(1,n-\frac{1}{2};\frac{n+1}{2};\frac{1}{d+1}\right)\right].\]
For large $n$ we have $D(n,d) = \frac{d^n-1}{d-1} \sim \frac{d^{n}}{d-1}$, so that $\lim\limits_{n\to\infty} \frac{E_{n,d}}{\sqrt{D(n,d)}}$ equals
\begin{align*}
\lim\limits_{n\to\infty}\frac{2^{n-1}\sqrt{d-1}\;\Gamma(n-\frac{1}{2})}{\sqrt{\pi}(d+1)^{\,n-\frac{1}{2}}\Gamma(n)}\; \left[2(n-1)~_2F_1\left(1,n-\frac{1}{2};\frac{3}{2};\frac{d-1}{d+1}\right)
	+~_2F_1\left(1,n-\frac{1}{2};\frac{n+1}{2};\frac{1}{d+1}\right)\right].
	\end{align*}
From description \cref{def_hypergeom} of $~_2F_1(a,b;c;x)$ we see that
	\begin{equation}\label{eq54}
	~_2F_1\left(1,n-\frac{1}{2};\frac{n+1}{2};\frac{1}{d+1}\right) = \sum_{k=0}^\infty \frac{(n-\frac{1}{2})_k}{(\frac{n+1}{2})_k}\; \frac{1}{(d+1)^k} \leq  \sum_{k=0}^\infty \left(\frac{2}{d+1}\right)^k \stackrel{d>1}{=} \frac{d+1}{d-1}.
	\end{equation}
By \cite[43:6:12]{atlas} we have that $\Gamma(n)\sim\Gamma(n-\frac{1}{2})\sqrt{n}$ for large $n$. Together with \cref{eq54} and $d>1$ this shows that 
	\begin{equation*}\frac{2^{n-1}\sqrt{d-1}\;\Gamma(n-\frac{1}{2})}{\sqrt{\pi}(d+1)^{\,n-\frac{1}{2}}\Gamma(n)}\; ~_2F_1\left(1,n-\frac{1}{2};\frac{n+1}{2};\frac{1}{d+1}\right)\stackrel{n\to\infty}{\longrightarrow } 0.
	\end{equation*}
and hence 
\begin{align*}
\lim\limits_{n\to\infty} \frac{E_{n,d}}{\sqrt{D(n,d)}}
&=\lim\limits_{n\to\infty}\frac{2^{n-1}\sqrt{d-1}\;\Gamma(n-\frac{1}{2})}{\sqrt{\pi}(d+1)^{\,n-\frac{1}{2}}\Gamma(n)}\; 2(n-1)~_2F_1\left(1,n-\frac{1}{2};\frac{3}{2};\frac{d-1}{d+1}\right)\\
	&=\lim\limits_{n\to\infty}\frac{2^{n}\sqrt{d-1}\;\Gamma(n-\frac{1}{2})}{\sqrt{\pi}(d+1)^{\,n-\frac{1}{2}}\Gamma(n-1)}\;~_2F_1\left(1,n-\frac{1}{2};\frac{3}{2};\frac{d-1}{d+1}\right).
	\end{align*}
Using \cref{hypergeom_as_beta}(1) we replace $~_2F_1\left(1,n-\frac{1}{2};\frac{3}{2};\frac{d-1}{d+1}\right)$ with the incomplete beta function so that 
\begin{align*}
\lim\limits_{n\to\infty} \frac{E_{n,d}}{\sqrt{D(n,d)}}
&=\lim\limits_{n\to\infty}\frac{2^{n}\sqrt{d-1}\;\Gamma(n-\frac{1}{2})}{\sqrt{\pi}(d+1)^{\,n-\frac{1}{2}}\Gamma(n-1)}\;\frac{(d+1)^{n-1}}{2^{n}}\;\frac{\sqrt{d+1}}{\sqrt{d-1}}\;B\left(\frac{1}{2},n-1,\frac{d-1}{d+1}\right)\\
&=\lim\limits_{n\to\infty}\frac{\Gamma(n-\frac{1}{2})}{\sqrt{\pi}\;\Gamma(n-1)}\;B\left(\frac{1}{2},n-1,\frac{d-1}{d+1}\right)
\end{align*}
\cref{beta_asymptotic} tells us that for large $n$ we have 
	\begin{equation*}B\left(\frac{1}{2},n-1,\frac{d-1}{d+1}\right) \sim \frac{\sqrt{\pi}\;\Gamma(n-1)}{\Gamma(n-\frac{1}{2})} \left[ \frac{1}{2} \;\mathrm{erfc}\left(-\omega\right) + b\right],
	\end{equation*}
where $\mathrm{erfc}(z)$ is the complementary error function and		
	\begin{equation*}\omega = \left[\frac{-1}{2}\ln \left(2n-1\right) - \frac{\ln\left(\frac{d-1}{d+1}\right)}{2} + (n-1)\ln\left(\frac{n-1}{n-\frac{1}{2}}\right) -(n-1)\ln\left(\frac{2}{d+1}\right)\right]^\frac{1}{2}
	\end{equation*}
and
\begin{equation*}b= \left(\frac{d-1}{2\pi (d+1) (n-1)}\right)^\frac{1}{2}\;\left(\frac{2(n-\frac{1}{2})}{(d+1)(n-1)}\right)^{n-1} \left(\frac{n-1}{\frac{1}{2}-(n-\frac{1}{2})\frac{d-1}{d+1}} + \frac{\sqrt{2n-1}}{\omega}\right) (1+\cO(n^{-1})).
\end{equation*}
We have $-\omega = \Theta(n)$ and hence $\omega \stackrel{n\to\infty}{\longrightarrow}  -\infty$ and $b\stackrel{n\to\infty}{\longrightarrow} 0$, which shows that
\begin{equation*}
\lim\limits_{n\to\infty} \frac{E_{n,d}}{\sqrt{D(n,d)}} =  \lim\limits_{n\to\infty} \left[\frac{1}{2}\;\mathrm{erfc}(-\omega) + b\right] = 1.
\end{equation*}
This finishes the proof.
\subsection{Proof of Theorem 1.1 (4)}
From \cref{main_thm} (2) we obtain
	\begin{equation*}
	\lim\limits_{d\to\infty} \frac{E_{n,d}}{\sqrt{D(n,d)}} =   \lim\limits_{d\to\infty}\frac{1}{\sqrt{D(n,d)}}  \;\frac{1}{\sqrt{\pi}}\;\frac{\Gamma(n-\frac{1}{2})}{\Gamma(n-1)} \frac{\sqrt{d}^{\,n}}{\sqrt{d+1}}\; \sum_{j=0}^{n-2} \binom{n-2}{j} \frac{ \left(-\frac{d-1}{d+1}\right)^j}{j+\frac{1}{2}} 
	\end{equation*}
For large $d$ we have $D(n,d)=\frac{d^n-1}{d-1} \sim d^{n-1}$ and $\frac{d-1}{d+1}\sim 1$, which shows that
	\begin{equation*}
	\lim\limits_{d\to\infty} \frac{E_{n,d}}{\sqrt{D(n,d)}} =  \frac{1}{\sqrt{\pi}}\;\frac{\Gamma(n-\frac{1}{2})}{\Gamma(n-1)}  \sum_{j=0}^{n-2} \binom{n-2}{j} \frac{ \left(-1\right)^j}{j+\frac{1}{2}}  =\frac{1}{\sqrt{\pi}}\;\frac{\Gamma(n-\frac{1}{2})}{\Gamma(n-1)}\; B\left(\frac{1}{2},n-1,1\right);
	\end{equation*}
the last equality by  \cite[58:4:3]{atlas}. By \cite[58:1:1]{atlas} we have $B\left(\frac{1}{2},n-1,1\right) = \sqrt{\pi}\; \frac{\Gamma(n-1)}{\Gamma(n-\frac{1}{2})}$, from which the claim follows.
\subsection{Proof of Theorem 1.1 (5)} We proceed as in the proof for \cite[Theorem 5.1]{real_eigenvalues}. Recall from \cref{main_step} that
 \begin{equation}\label{eq20}
 E_{n,d}\;=\mean\limits_{f\sim N((\cH_d^\HR)^n)} \,\#_\HR(f) = \mean\limits_{\lambda \sim N(0,1)} F_{n-1,d}(\lambda)
 \end{equation}
We first compute the generating function of the $F_{n,d}(\lambda)$. Fix $d\geq 1$,  $\lambda\in \HR$ and $- \frac{1}{\sqrt{d}}< z< \frac{1}{\sqrt{d}}$. In the definition~\cref{F_nd} of~$F_{n,d}(\lambda)$ replace the gamma functions by the respective integrals from \cref{gamma_fct}. Since the integrands in the gamma function are all positive, we may apply Fubini's theorem to interchange summation and integration. This yields
 \begin{align*}
& \sum_{n=0}^\infty F_{n,d}(\lambda)\,z^n \\
 &=1+\sum_{n=1}^\infty F_{n,d}(\lambda)\,z^n \\
 &= 1+ e^\frac{\lambda^2}{2d}
\int_{t=\frac{\lambda^2}{d}}^\infty e^{-t}  \left[\sum_{n=1}^\infty \frac{\sqrt{d}^{\,n} }{\Gamma(n)}\, t^{n-1}  z^n\right] \,\d t
+   \int_{t=0}^\frac{\lambda^2}{2d}  e^{-t}\left[\sum_{n=1}^\infty \frac{\sqrt{d}^n}{\Gamma(n)}\,2^{n-1}\left(\frac{\lambda^2}{2d}\right)^\frac{n}{2}  t^{\frac{n}{2}-1}\,z^n\right]\, \d t\\
&=1+ e^\frac{\lambda^2}{2d}\;z\sqrt{d}\,
\int_{t=\frac{\lambda^2}{d}}^\infty  e^{-t(1-z\sqrt{d})}\;\d t
+   \frac{1}{\sqrt{2}}\,z\,\norm{\lambda}  \int_{t=0}^\frac{\lambda^2}{2d}  \frac{e^{-t+ z\sqrt{2t}\,\norm{\lambda}}}{\sqrt{t}}\;\d t\\
&\hspace{0.5cm} (\text{in the right hand integral we now substitute } s^2=t)\\
&=1+ e^\frac{\lambda^2}{2d}\;z\sqrt{d}\;
\frac{e^{-\frac{\lambda^2}{d}(1-z\sqrt{d})}}{1-z\sqrt{d}} 
+   \frac{\sqrt{\pi}}{\sqrt{2}} \;z\,\norm{\lambda} \; e^\frac{z^2\lambda^2}{2} \left[ \frac{2}{\sqrt{\pi}}\,\int_{s=0}^\frac{\norm{\lambda}}{\sqrt{2d}}  e^{-\left(s-\frac{z\,\norm{\lambda}}{\sqrt{2}}\right)^2}\;\d s\right]\\
&=1 + \frac{z\sqrt{d}}{1-z\sqrt{d}}\; e^\frac{-\lambda^2\,(1-2z\sqrt{d})}{2d} +  \frac{\sqrt{\pi}}{\sqrt{2}} \;z\,\norm{\lambda} \; e^\frac{z^2\lambda^2}{2} \left[\mathrm{erf}\left(\frac{z\,\norm{\lambda}}{\sqrt{2}}\right) + \mathrm{erf}\left(\frac{(1-z\sqrt{d})\,\norm{\lambda}}{\sqrt{2d}}\right)\right],
\end{align*}
where $\mathrm{erf}(x):=\frac{2}{\sqrt{\pi}}\int_{0}^x e^{-t^2}\d t$ denotes the error function. By \cref{eq20} we have 
\[\sum_{n=1}^\infty E_{n,d} \;z^n = z\,\mean\limits_{\lambda \sim N(0,1)} \sum_{n=0}^\infty F_{n,d}(\lambda) \,z^n\]
and hence $\sum_{n=1}^\infty E_{n,d} \;z^n = z\,(s_1 +s_2+s_3),$ where 
\begin{equation*}
s_1 = \mean\limits_{\lambda \sim N(0,1)} 1 = 1,\quad
s_2 =  \frac{\sqrt{d}\,z}{1-\sqrt{d}\,z}\;\mean\limits_{\lambda \sim N(0,1)} e^\frac{-\lambda^2\,(1-2\sqrt{d}\,z)}{2d} = \frac{zd}{(1-\sqrt{d}z)\;\sqrt{1+d-2z\sqrt{d}}}
\end{equation*}
(for the this we used $\mean\limits_{\lambda \sim N(0,1)} e^\frac{-a\lambda^2}{2} = \frac{1}{\sqrt{a+1}}$), and
\begin{equation*}
s_3 =\frac{\sqrt{\pi}}{\sqrt{2}} \;z\mean\limits_{\lambda \sim N(0,1)}\,\norm{\lambda} e^\frac{z^2\lambda^2}{2} \left[\mathrm{erf}\left(\frac{z\norm{\lambda}}{\sqrt{2}}\right) + \mathrm{erf}\left(\frac{(1-z\sqrt{d})\norm{\lambda}}{\sqrt{2d}}\right)\right] \stackrel{\text{calc}}{=} \frac{z^2}{1-z^2} + \frac{z(1-z\sqrt{d})}{(1-z^2) \,\sqrt{1+d-2z\sqrt{d}}}.
\end{equation*}
Thus 
	\[s_1+s_2+s_3\stackrel{\text{calc}}{=}\frac{1-z\sqrt{d}  +z\sqrt{d-2z\sqrt{d} +1}}{(1-z^2)(1-z\sqrt{d})},\]
which finishes the proof.
\begin{rem} The symbol $\stackrel{\text{calc}}{=}$ indicates that we computed the respective equality with \textsc{maple 18} \cite{maple}.\end{rem}
\section{Gaussian tensors and gaussian polynomial systems}
Recall that the tensor $A=(A_{i_0,i_1,\ldots,i_{d}}) \in \left(\HR^{n}\right)^{\otimes (d+1)}$ is said to be gaussian, if $A_{i_0,i_1,\ldots,i_{d}}\sim N(0,1)$. In fact, the variance is defined in the way, such that the coefficients of $f_A(X)=AX^d$ in the Bombieri-Weyl basis (see \cref{se:geo-framework}) are i.i.d $N(0,1)$-random variables. This is summarized in the following lemma.
\begin{lemma}\label{additional_lemma}
Let $A\in \left(\HR^{n}\right)^{\otimes (d+1)}$ be gaussian. Then $f_A\sim N((\cH_d^\HR)^n)$.
\end{lemma}
\begin{proof}
Let $A\in\HR^{n^{d+1}}$ be gaussian and write $f_A(X)=(f_1(X),\ldots,f_n(X))$. We have to show that every coefficient of each $f_i$ in the Bombieri-Weyl basis $\{\sqrt{\binom{d}{\alpha}}\, X^\alpha\}$ is $N(0,1)$-distributed. Fix $1\leq j\leq n$ and suppose that
	\begin{equation}\label{5.1}
	f_j(X)=\sum_{\alpha_1+\ldots+\alpha_n = d} \lambda_\alpha \sqrt{\binom{d}{\alpha}}\, X^\alpha.
	\end{equation}
By definition we have $f_A(X)=AX^d$, which, by \cref{1}, implies that 
	\begin{equation}\label{5.2}
	f_j(X)=\sum_{1\leq i_1,\ldots,i_d\leq n} A_{j,i_1,\ldots,i_d} \;\; X_{i_1}\ldots X_{i_d}.
	\end{equation}
Comparing \cref{5.1} and \cref{5.2} reveals that for each $\alpha$ we have\[\lambda_\alpha\, \sqrt{\binom{d}{\alpha}} = \sum\limits_{\substack{(i_1,\ldots,i_d): \\ X^\alpha = X_{i_1}\ldots X_{i_d}}} A_{j,i_1,\ldots,i_d}.\]
Applying the rule of summation of normal distributed random variables from \cref{sum_of_normal} (2) to this equation shows that
	\[\lambda_\alpha\, \sqrt{\binom{d}{\alpha}} \sim N(0,\sigma^2), \text{ where } \sigma^2 = \# \cset{(i_1,\ldots,i_d)}{X^\alpha = X_{i_1}\ldots X_{i_d}} = \binom{d}{\alpha}.\]
Hence, $\lambda_\alpha\sim N(0,1)$, by \cref{sum_of_normal} (1). This finishes the proof.
	\end{proof}
{
\bibliographystyle{plain}
\bibliography{literature}}
\end{document}